\documentclass[leqno]{amsproc}
\usepackage{times}
\usepackage{amsfonts,amssymb,amsmath,amsgen,amsthm}
\usepackage{hyperref}

\theoremstyle{plain}
\newtheorem{theorem}{Theorem}[section]
\newtheorem{definition}[theorem]{Definition}

\newtheorem{lemma}[theorem]{Lemma}
\newtheorem{corollary}[theorem]{Corollary}
\newtheorem{proposition}[theorem]{Proposition}

\theoremstyle{remark}

\def\R{{\mathbf R}}
\def\N{{\mathbf N}}
\def\O{\mathcal O}

\def\({\left(}
\def\){\right)}
\def\<{\left\langle}
\def\>{\right\rangle}
\def\le{\leqslant}
\def\ge{\geqslant}

\def\Eq#1#2{\mathop{\sim}\limits_{#1\rightarrow#2}}
\def\Tend#1#2{\mathop{\longrightarrow}\limits_{#1\rightarrow#2}}

\def\d{{\partial}}
\def\eps{\varepsilon}
\def\l{\lambda}

\def\si{{\sigma}}

\numberwithin{equation}{section}

\begin{document}

\title[On  Schr\"odinger
  equations with modified dispersion]{On  Schr\"odinger
  equations with modified dispersion}   
\author[R. Carles]{R\'emi Carles}
\address{Univ. Montpellier~2\\Math\'ematiques
\\CC~051\\F-34095 Montpellier}
\address{CNRS, UMR 5149\\  F-34095 Montpellier\\ France}
\email{Remi.Carles@math.cnrs.fr}

\begin{abstract}
 We consider the nonlinear Schr\"odinger equation with a modified
 spatial dispersion, given either by an homogeneous Fourier
 multiplier, or by a bounded Fourier multiplier. Arguments based on
 ordinary differential equations yield ill-posedness results which are
 sometimes sharp. When the
 Fourier multiplier is bounded, we infer that no Strichartz-type
 estimate improving on Sobolev embedding is available. Finally, we
 show that when the symbol is bounded,
 the Cauchy problem may be 
 ill-posed in the case of critical
 regularity, with arbitrarily small initial data. The same is true
 when the symbol is homogeneous of degree one, where scaling arguments
 may not even give the right critical value. 
\end{abstract}
\thanks{This work was supported by the French ANR project
  R.A.S. (ANR-08-JCJC-0124-01).}
\maketitle

\section{Introduction}
\label{sec:intro}
For $(t,x)\in \R_+\times \R^d$, we consider 
\begin{equation}
  \label{eq:gen}
  i\d_t u + P(D)u = \l |u|^{2\si}u\quad ;\quad u_{\mid t=0}=u_0,
\end{equation}
where $D=-i\nabla_x$, $P:\R^d\to \R$,
$\l\in \R$ and $\si>0$. Since the Fourier multiplier $P$ is
real-valued, the free flow ($\l=0$) 
generates a unitary group on $\dot H^s(\R^d)$, $s\in \R$: 
\begin{equation*}
  S(t) = e^{-it P(D)}. 
\end{equation*}
We consider two cases:
\begin{itemize}
\item $P$ is homogeneous, $P(\mu \xi)= \mu^m P(\xi)$, for all $\xi\in \R^d$,
  $\mu>0$, with $m\ge 1$. 
\item $P$ is bounded, $P\in L^\infty(\R^d)$.
\end{itemize}
The first case includes the standard Schr\"odinger operator ($P(\xi)=-|\xi|^2$),
and the fourth-order Schr\"odinger operator ($P(\xi)=|\xi|^4$), 
studied for instance in \cite{BAKS00,Pau07,Pau09}. Smoothing effects and
dispersive estimates have
been established for rather general Fourier multipliers $P$ in
\cite{BAS99,KPV91}. The case $d=1$,
$P(\xi)=\xi^{2j+1}$, $j\in \N$, has been studied initially in
\cite{KPV94}; the case $d=2$, with $P$ a polynomial of degree $m=3$
has been studied in \cite{BAKS03}, 
revealing different dispersive phenomena according to the precise structure of
$P$. The  case 
$m\in 2\N$, with $P$ elliptic and $\nabla^2P$ non-degenerate outside
$\{\xi=0\}$, appears as a particular 
case of the framework 
in \cite{CuGu07}. It is shown there that if $s>\max(0,s_0)$, then the
Cauchy problem \eqref{eq:gen} is locally well-posed in $H^s(\R^d)$,
where 
\begin{equation}\label{eq:sc}
  s_0 = \frac{d}{2}- \frac{m}{2\si}. 
\end{equation}
This index corresponds to the one given by scaling arguments: if $u$
solves \eqref{eq:gen}, then for any $\Lambda>0$, $u_\Lambda:(t,x)\mapsto 
\Lambda^{m/(2\si)}
u(\Lambda^mt,\Lambda x)$ solves the same equation. The value of $s$ for which
the $\dot H^s(\R^d)$-norm is invariant under $u\mapsto  u_\Lambda$ is
$s=s_0$. We will see in \S\ref{sec:crit} that this scaling argument 
may not yield the sharp Sobolev regularity: if $m=1$, the Cauchy
problem \eqref{eq:gen} may be strongly ill-posed in $H^s(\R^d)$ for
all $s\le d/2$.  

In \cite{CuGu07}, the proof of local well-posedness uses dispersive
and Strichartz estimates for $S$, established in \cite{Cui05} for $d=1$, 
and in \cite{Cui06} for $d\ge 2$. Note that in the case $d=1$, 
dispersive and Strichartz estimates for $S$ are proved in
\cite{HaNa08} for $P(\xi)=|\xi|^m$ and  $m\ge 2$ (not necessarily
an integer).   
By resuming arguments similar to those presented in
\cite{BGTENS,CCT2}, we prove that in this framework, the index $s_0$
is sharp, in the sense that the 
nonlinear flow map fails to be uniformly continuous at the origin in
$H^s(\R^d)$, if $s<s_0$. This property has been established in
\cite{Pau09} for the case $P(\xi)=|\xi|^4$ with
$(d,\si)=(3,1)$. However,  the index $s_0$ may not correspond to the
critical Sobolev regularity (see \S\ref{sec:crit}).
\smallbreak

The second case, where $P$ is bounded, is motivated by the results
presented in \cite{DeFa09}, where $P(\xi)=-\frac{1}{h}\arctan(h|\xi|^2)$ is
considered to construct numerical approximations of the solution to 
the linear Schr\"odinger equation
\begin{equation*}
  i\d_t u + \Delta u = V(x)u, 
\end{equation*}
and $0<h\ll 1$ denotes the time step. We will see below that in such a
framework, no Strichartz estimate is available, even if one is ready
to pay some loss of derivative. Another example of bounded symbol one
may think of is 
\begin{equation*}
  i\d_t u + \Delta \(1-\Delta\)^{-1}u =\l |u|^{2\si}u.
\end{equation*}
In these two examples, $P$ is elliptic. We will see however that no
Strichartz estimate (better than Sobolev embedding) is available
there, and that the critical regularity is $s_c=d/2$.  
\subsection{Norm inflation}
\label{sec:inflation}
Our result in this direction is:
\begin{theorem}\label{theo}
  Let $d\ge 1$, $\l\in \R$, $\si>0$. Assume that either $\si$ is an
  integer, or that there exists an integer $r$ such that 
$2\si\ge r>d/2$. \\
$1.$ Suppose that $P$ is $m$-homogeneous, with $m\ge  1$, and
denote
$s_0 = d/2- m/(2\si)$. Suppose that $s_0>0$ and let $0<s<s_0$. 
There exists
a family $(u_0^h)_{0<h \le 
  1}$ in ${\mathcal S}({\R}^d)$ with 
\begin{equation*}
  \|u_0^h\|_{H^{s}({\R}^d)} \to 0 \text{ as
  }h \to 0, 
\end{equation*}
and a solution $u^h$ to
\eqref{eq:gen} and $0<t^h \to 0$, such that: 
\begin{equation*}
  \|u^h(t^h)\|_{H^{s}(\R^d)} \to +\infty \text{ as }h \to
 0.
\end{equation*}
$2.$ If $P$ is bounded, then the above conclusion remains
true for any $0<s<d/2$. 
\end{theorem}
Theorem~\ref{theo} is proved in \S\ref{sec:theo}, by adapting the
ordinary differential equation mechanism used in, e.g.,
\cite{BGTENS,CCT2}. However, the critical Sobolev regularity
$s_c$ may satisfy $s_c>s_0$ (see \S\ref{sec:crit}): in view of \cite{CuGu07}
and Theorem~\ref{theo}, we have $s_c=s_0$
at least when $P(\xi)=\mu |\xi|^m$, $\mu\in \R\setminus\{0\}$, $m\in
2\N\setminus\{0\}$. We also refer to \cite{BoSa10} where a different
result concerning the 
lack of well-posedness is established for a broad variety of
dispersive equations, even in the linear case. 
\subsection{Absence of Strichartz estimates}
\label{sec:stri}

In this paragraph, we focus our discussion on the case where $P$ is
bounded. For $s>d/2$,
$H^s(\R^d)$ being an algebra, local well-posedness in $H^s(\R^d)$ is
straightforward, provided that the nonlinearity is sufficiently smooth
(see e.g. \cite{TaoDisp}). Therefore, the critical 
threshold is $s_c=d/2$, and from Theorem~\ref{theo}, no dispersive
property is present to decrease this number. More precisely, no
Strichartz estimate is available for $S(\cdot)$, even if one is ready
to pay some loss of regularity which is not worse than the result
provided by Sobolev embedding. 
To state this property precisely,
we recall the standard definition.
\begin{definition}\label{def:adm}
 A pair $(p,q)\not =(2,\infty)$ is admissible if $p\ge 2$, $q\ge  2$,
 and 
$$\frac{2}{p}= d\left( \frac{1}{2}-\frac{1}{q}\right).$$
\end{definition}
By Sobolev embedding, for all $(p,q)$ (not necessarily admissible)
with $2\le q<\infty$, there exists $C>0$ such that for all $u_0\in
H^{d/2-d/q}(\R^d)$, and all finite time interval $I$,
\begin{align*}
    \|S(\cdot)u_0\|_{L^p(I;L^q(\R^d))}&\le
    C\|S(\cdot)u_0\|_{L^p(I;H^{d/2-d/q}(\R^d))}\\
    &\le C\|u_0\|_{L^p(I;H^{d/2-d/q}(\R^d))} =
    C|I|^{1/p}\|u_0\|_{H^{d/2-d/q}(\R^d)}. 
  \end{align*}
When $P$ is bounded, this estimate cannot be improved:
\begin{corollary}\label{cor:noStri}
  Let $d\ge 1$, and $P\in L^\infty(\R^d;\R)$. Suppose that there
  exist an admissible pair $(p,q)$, an index $k\in \R$, a time
  interval $I\ni 0$, $|I|>0$, and  a constant $C>0$ such that
  \begin{equation*}
    \|S(\cdot)u_0\|_{L^p(I;L^q(\R^d))}\le C\|u_0\|_{H^k(\R^d)},\quad
    \forall u_0\in H^k(\R^d). 
  \end{equation*}
Then necessarily, $k\ge 2/p= d/2-d/q$. 
\end{corollary}
The fact that no standard Strichartz
estimate (with no loss) is available for $S(\cdot)$ is rather clear,
since the dispersion relation is given by $\tau=P(\xi)$, and defines a
characteristic variety which is bounded in $\tau$. However, one could
expect the existence of Strichartz estimates with loss of regularity,
in the same fashion as in \cite{Ant08,BSS08,BGT} (where the geometric framework
--- the space variable belongs to a compact manifold --- rules out the
existence of the standard 
dispersive properties). The fact that this is not so is
a rather direct consequence of Theorem~\ref{theo} (where $\si>0$ is
arbitrary), and of the argument given in \cite{BGT} to prove
Proposition~3.1. It may seem surprising to prove
Corollary~\ref{cor:noStri} as a consequence of a nonlinear analysis;
we insist on the fact that the proof of 
Theorem~\ref{theo} is rather simple, and the deduction of
Corollary~\ref{cor:noStri} involves another nonlinear result, whose
proof is also quite short (see Proposition~\ref{prop:LWP} below). 

\subsection{Critical cases}
\label{sec:crit}

 In \cite{CW90},  local well-posedness in $H^{s_c}(\R^d)$ for small
  data is established for Equation~\eqref{eq:gen} in the case
  $P(\xi)=-|\xi|^2$, where
  \begin{equation*}
    s_c=\frac{d}{2}-\frac{1}{\si}
  \end{equation*}
coincides with $s_0$ in that case, since $m=2$. 
 In \cite{NaOz98}, local well-posedness in $H^{d/2}(\R^d)$ for small
  data is established for the same operator, with nonlinearities which
  are  allowed to grow exponentially. In these two papers, the proof
  uses Strichartz estimates (in Besov 
  spaces). On the other hand, when $P$ is bounded, the Cauchy problem
  may be ill-posed in $H^{d/2}(\R^d)$, even for nonlinearities growing
  algebraically.  
Moreover, when $P$ is $m$-homogeneous with $m=1$, the critical Sobolev
regularity may not be  $s_0$, but $s_c=d/2>s_0$, with ill-posedness
for $s=s_c$. 

  \begin{proposition}\label{prop:ill}
   Let $\l\in \R\setminus \{0\}$, $\si>0$.  Assume that
   either $\si$ is an 
  integer, or that there exists an integer $r$ such that 
$2\si\ge r>d/2$. In 
   either of the two cases: 
    \begin{itemize}
    \item $P(\xi)=c\cdot \xi$, $c\in \R^d$, or
\item $P$ is constant, 
    \end{itemize}
for all $\delta >0$, there exists $u_0\in H^{d/2}(\R^d)$ with
$\|u_0\|_{H^{d/2}(\R^d)}\le \delta$ such that 
\eqref{eq:gen} has a unique solution $u\in C(\R_+;{\mathcal
  D}'(\R^d))$, and 
for any $t>0$,  $u(t,\cdot)\not \in H^{d/2}(\R^d)$.  
  \end{proposition}
For comparison with other results, note that in the first case, $P$ is
not elliptic if $d\ge 2$. In the second 
case, $P\not =0$ is elliptic, but $\nabla^2P=0$ is obviously
degenerate. 
\smallbreak

Unlike what happens in the presence of Strichartz estimates
(\cite{CW90,NaOz98}), this result yields examples where local
well-posedness fails in 
the critical case $s=s_c$, even for small data.  We prove
Proposition~\ref{prop:ill} in \S\ref{sec:ill}: we present the cases
$d=2$, $\si>0$, and $d\ge 1$, $\si=1/2$,  only, for the convenience of
the exposition, but the 
argument can be extended to any space dimension, up to more intricate
computations.
\smallbreak

In the case $m=1$ (at least), the mere assumption that $P$
is $m$-homogeneous is not enough to characterize the critical Sobolev
regularity in \eqref{eq:gen}, or the existence of Strichartz
estimates. Indeed, the symbol
$P(\xi)=|\xi|$ corresponds to the wave equation, for which Strichartz
estimates are available when $d\ge 2$, and so $s_c<d/2$. See
e.g. \cite{GV95}. This suggests that also when
$m>1$, the value $s_0$ may 
not be sharp when $P$ is not proportional to $|\xi|^m$, but for
instance of the form $|\xi|^{m-1}c\cdot \xi$, $c\in \R^d$, or more
generally when $P$ is not elliptic; see also \cite{BAKS03,BAS99,KPV91}
for remarks in this direction.

\section{Proof of Theorem~\ref{theo}}
\label{sec:theo}

The proof of Theorem~\ref{theo} proceeds along the same lines as in
\cite{BGTENS,CCT2} (see also \cite{TaoDisp}). Fix $s$ as in
Theorem~\ref{theo}. Consider initial data of
the form
\begin{equation*}
  u^h_0(x) = h^{s-d/2} \kappa^ha_0\(\frac{x}{h}\), 
\end{equation*}
with $0<h\ll 1$,  $a_0(x)=e^{-|x|^2}$, and
\begin{equation*}
  \kappa^h = \(\log \frac{1}{h}\)^{-\theta}
\end{equation*}
for some $\theta>0$ to be fixed later. We have
$\|u_0^h\|_{H^s(\R^d)} \Tend h 0 0 $. 
Introduce the scaling
\begin{equation*}
  \psi(\tau,y) = h^{d/2-s}u^h\( h^{2+\alpha}\tau,hy\),
\end{equation*}
for some $\alpha$ to be precised later:
\begin{equation*}
  \|\psi(\tau)\|_{\dot H^s(\R^d)}  = 
  \|u^h\(h^{2+\alpha}\tau\)\|_{\dot H^s(\R^d)}. 
\end{equation*}
Denote  $\eps = h^{2\si(d/2-s)-2-\alpha}$. The function $\psi$ solves the
Cauchy problem
\begin{equation}\label{eq:psi}
      i\eps\d_\tau \psi + h^{2\si(d/2-s)}P\(h^{-1}D_y\)\psi
  = \l|\psi|^{2\si}\psi;\quad
 \psi_{\mid \tau=0} = \kappa^ha_0. 
\end{equation}
\subsection{Choice of $\alpha$}
\label{sec:homo}

When $P$ is $m$-homogeneous, Equation~\eqref{eq:psi} simplifies to 
\begin{equation*}
  i\eps\d_\tau \psi + h^{2\si(d/2-s)- m}P(D_y)\psi = \l|\psi|^{2\si}\psi;\quad
 \psi_{\mid \tau=0} = \kappa^ha_0.
\end{equation*}
For $\omega>0$, we set
\begin{equation*}
  2+\alpha =
  \frac{1}{m+\omega}\(\(m-1+\omega\)2\si\(\frac{d}{2}-s\)+m\),
\end{equation*}
in which case we have:
\begin{equation*}
  \eps = h^{2\si(s_0-s)/(m+\omega)}. 
\end{equation*}
Therefore, $\eps\to 0$ as $h\to 0$ since $s<s_0$. 
We also compute
\begin{equation*}
  h^{2\si(d/2-s)- m} = \eps^{m+\omega}.
\end{equation*}
\smallbreak

When $P$ is bounded, we consider $2+\alpha =
\si(d/2-s)$. Therefore,
\begin{equation*}
  2\si\(\frac{d}{2}-s\)-2-\alpha>0\quad \text{(hence }\eps\to 0
  \text{ as }h\to 0\text{)},\quad \text{and }2+\alpha>0.
\end{equation*}
\subsection{The ODE approximation}
\label{sec:ode}
Introduce the solution to 
\begin{equation*}
  i\eps \d_\tau \varphi = \l|\varphi|^{2\si}\varphi;\quad
 \varphi_{\mid \tau=0} = \kappa^h a_0. 
\end{equation*}
It is given by
\begin{equation*}
  \varphi(\tau,y) = \kappa^h a_0(y)\exp\( -i\l \frac{\tau}{\eps}
  \(\kappa^h\)^{2\si}|a_0(y)|^{2\si}\). 
\end{equation*}
Since $a_0$ is a Gaussian, $\varphi\in C^\infty(\R\times\R^d)$
regardless of $\si>0$, and
for any $r\ge 0$,  
\begin{equation}\label{eq:estphi}
  \|\varphi(\tau)\|_{H^r(\R^d)} \lesssim \(\kappa^h\)^{1+
    2\si r} \(\frac{\tau}{\eps}\)^{r} + \kappa^h. 
\end{equation}

\begin{lemma}\label{lem:ode}
Let $r>d/2$ be an integer, with in addition $r\le 2\si$ if
$\si\not\in\N$. In either 
of the cases of Theorem~\ref{theo}, we can find 
$\delta>0$ independent of $\theta>0$ such that
  \begin{equation*}
   \sup_{0\le \tau \le \eps \(\log
    \frac{1}{\eps}\)^{\delta}}
 \|\psi(\tau)-\varphi(\tau)\|_{H^r(\R^d)}\Tend \eps 0 0 .
  \end{equation*}
\end{lemma}
\begin{proof}
  Denote by $w=\psi-\varphi$ the error. It solves
  \begin{align*}
    i\eps\d_\tau w + h^{2\si(d/2-s)}P\(h^{-1}D_y\)w &=
    h^{2\si(d/2-s)}P\(h^{-1}D_y\)\varphi \\
    &\quad +\l\(|w+\varphi|^{2\si}(w+\varphi)-|\varphi|^{2\si}\varphi \),
  \end{align*}
with $w_{\mid \tau=0}=0$. Using the facts that $P$ is real-valued,
$z\mapsto |z|^{2\si}z$ is sufficiently smooth,  and
$H^r(\R^d)$ is an algebra, we find
\begin{align*}
  \|w(\tau)\|_{H^r(\R^d)} &\lesssim \frac{1}{\eps}\int_0^\tau \left\|
    h^{2\si(d/2-s)}P\(h^{-1}D_y\)\varphi (\tau')\right\|_{H^r(\R^d)}d\tau' \\
&\quad+
\frac{1}{\eps}  \int_0^\tau \(\|w(\tau')\|_{H^r(\R^d)}^{2\si} +
  \|\varphi(\tau')\|_{H^r(\R^d)}^{2\si}
  \)\|w(\tau')\|_{H^r(\R^d)}d\tau'.  
\end{align*}
In the case where $P$ is $m$-homogeneous, we have
\begin{align*}
  \left\|
    h^{2\si(d/2-s)}P\(h^{-1}D_y\)\varphi (\tau')\right\|_{H^r(\R^d)}
  & \lesssim  
h^{2\si(d/2-s)-m}\left\|
    \varphi (\tau')\right\|_{H^{r+m}(\R^d)}\\
&\lesssim h^{2\si(s_0-s)}\left\|
    \varphi (\tau')\right\|_{H^{r+m}(\R^d)}. 
\end{align*}
In the case where $P$ is bounded, we have 
\begin{align*}
  \left\|
    h^{2\si(d/2-s)}P\(h^{-1}D_y\)\varphi (\tau')\right\|_{H^r(\R^d)}
  & \lesssim  
h^{2\si(d/2-s)}\left\|
    \varphi (\tau')\right\|_{H^{r}(\R^d)}\\
&\lesssim h^{2\si(s_0-s)}\left\|
    \varphi (\tau')\right\|_{H^{r}(\R^d)}, 
\end{align*}
where we set $s_0=d/2$ in this case. 
\smallbreak

In both cases, we check that there exists $\beta>0$ (independent of
$\theta$) such that 
\begin{equation*}
  h^{2\si(s_0-s)} =\eps^{1+\beta}.
\end{equation*}
It is given by the formula
\begin{equation*}
  \beta = \frac{2\si(s_0-d/2)+2+\alpha}{2\si(d/2-s)-2-\alpha}. 
\end{equation*}
In the homogeneous case, this formula becomes $\beta= m-1+\omega$, and
when $P$ is bounded, $\beta=1$.
Therefore, in view of
\eqref{eq:estphi} and since $\kappa^h\le 1$, there exist $\beta,\gamma
>0$ such that 
\begin{equation*}
  \left\|
    h^{2\si(d/2-s)}P\(h^{-1}D_y\)\varphi (\tau')\right\|_{H^r(\R^d)}
  \le \eps^{1+\beta}\( 
  \(\frac{\tau}{\eps}\)^\gamma + 1\) .
\end{equation*}
So long as $\|w(\tau)\|_{H^r(\R^d)}\le 1$, with $\tau$ as above, we
infer from \eqref{eq:estphi}:
\begin{equation*}
  \|w(\tau)\|_{H^r(\R^d)}\lesssim \int_0^\tau \eps^\beta\( 
  \(\frac{\tau'}{\eps}\)^\gamma + 1\)d\tau' + \frac{1}{\eps}\int_0^\tau
  \(1+\(\frac{\tau'}{\eps}\)^r\)\|w(\tau')\|_{H^r(\R^d)}d\tau' . 
\end{equation*}
Gronwall's Lemma yields
\begin{equation*}
 \|w(\tau)\|_{H^r(\R^d)}  \lesssim  \eps^\beta\( 
  \(\frac{\tau}{\eps}\)^\gamma + 1\)e^{C\tau/\eps +C(\tau/\eps)^{r+1}}\lesssim
  \eps^\beta e^{2C\tau/\eps+C(\tau/\eps)^{r+1} }.
\end{equation*}
By choosing
$\delta>0$ sufficiently small, the right hand side is controlled by,
say, $\eps^{\beta/2}$, for all $0\le \tau\le \eps
\(\log\frac{1}{\eps}\)^{\delta}$. The condition $\|w(\tau)\|_{H^r(\R^d)}\le 1$ is
verified for such times $\tau$, provided that
$\eps$ is sufficiently small, and the lemma follows. 
\end{proof}

\subsection{Conclusion}
Let $r>d/2$ as in Lemma~\ref{lem:ode}. With $\delta>0$ given by
Lemma~\ref{lem:ode}, we have, 
since $s<d/2$: 
\begin{align*}
  \left\|u^h\(h^{2+\alpha}\eps\(\log
    \frac{1}{\eps}\)^\delta\)\right\|_{H^s(\R^d)}&\ge
\left\|\varphi\(\eps\(\log
    \frac{1}{\eps}\)^\delta\)\right\|_{H^s(\R^d)}\\
&  - C\left\|\varphi\(\eps\(\log
    \frac{1}{\eps}\)^\delta\)- \psi\(\eps\(\log
    \frac{1}{\eps}\)^\delta\)\right\|_{H^r(\R^d)}.
\end{align*}
On the other hand, similar to \eqref{eq:estphi}, we have:
\begin{equation*}
  \|\varphi(\tau)\|_{H^s(\R^d)} \gtrsim \(\kappa^h\)^{1+
    2\si s} \(\frac{\tau}{\eps}\)^{s} -C \kappa^h, 
\end{equation*}
and so,
\begin{equation*}
 \left\|\varphi\(\eps\(\log
    \frac{1}{\eps}\)^\delta\)\right\|_{H^s(\R^d)}\ge C
  \(\log\frac{1}{\eps}\)^{s\delta -\theta -2\si \theta s}-o(1). 
\end{equation*}
For $\theta>0$ sufficiently small, $s\delta -\theta -2\si \theta s>0$,
and Lemma~\ref{lem:ode} yields
\begin{equation*}
  \left\|u^h\(h^{2+\alpha}\eps\(\log
    \frac{1}{\eps}\)^\delta\)\right\|_{H^s(\R^d)}\Tend h 0
  +\infty. 
\end{equation*}
Theorem~\ref{theo} follows, with
\begin{equation*}
  t^h = h^{2+\alpha}\eps\(\log
    \frac{1}{\eps}\)^\delta= C h^{2\si(d/2-s)} \(\log
    \frac{1}{h}\)^\delta\Tend h 0 0 . 
\end{equation*}
\section{Proof of Corollary~\ref{cor:noStri}}
\label{sec:cor}

We argue by contradiction, by using a slight generalization of
\cite[Proposition~3.1]{BGT}. 
\begin{definition}[From \cite{KPV01}]
\label{def:WP}
Let $s \in \R$. The Cauchy problem \eqref{eq:gen}
is well posed in $H^{s}(\R^d)$  if, for all bounded
subset $B\subset H^{s}(\R^d)$, there exist  
$T>0$ and a Banach space $X_T$ continuously embedded into
$C([0,T];H^s(\R^d))$ 
such that for all $u_0\in H^s(\R^d)$, \eqref{eq:gen}  
has a unique solution $u \in X_T$,
and the mapping $u_0
  \mapsto u$ is uniformly continuous from $ (B,\|
  \cdot \|_{H^s})$ to $C([0,T];H^s(\R^d))$.
\end{definition}
\begin{proposition}\label{prop:LWP}
  Let $d\ge 1$, $P:\R^d\to \R$. Suppose that there
  exist an admissible pair $(p,q)$, an index $k<2/p=d/2-d/q$, $T_0>0$,
  and  a constant $C>0$ such that 
  \begin{equation}\label{eq:stri}
    \|S(\cdot)u_0\|_{L^p([0,T_0];L^q(\R^d))}\le C\|u_0\|_{H^k(\R^d)},\quad
    \forall u_0\in H^k(\R^d). 
  \end{equation}
Then for all 
\begin{equation*}
k+\frac{d}{q}<s<\frac{d}{2},\quad   0<\si<\frac{p}{2}, 
\end{equation*}
 the Cauchy problem for \eqref{eq:gen}
is well posed in $H^{s}(\R^d)$. 
\end{proposition}
Since in Theorem~\ref{theo}, we can always consider $\si=1$, 
Theorem~\ref{theo} and Proposition~\ref{prop:LWP} imply
Corollary~\ref{cor:noStri} in the non-endpoint case 
$p>2$. The endpoint case then follows by interpolation with the case
$(p,q)=(\infty,2)$: if an endpoint Strichartz estimate (with some
loss) was available, then an non-endpoint would be as well.  
\begin{proof}
For $0<T\le T_0$, introduce
  \begin{equation*}
    X_T = C\([0,T];H^s(\R^d)\)\cap L^p\([0,T];W^{\ell,q}(\R^d)\), 
  \end{equation*}
where  $\ell = s-k$. By assumption, $\ell >d/q$, so we have
\begin{equation*}
  X_T\subset L^p\([0,T];L^\infty(\R^d)\). 
\end{equation*}
This space is equipped with the norm
\begin{equation*}
  \|u\|_{X_T} = \sup_{0\le t\le T}\|u(t)\|_{H^s(\R^d)} + \left\|
    (1-\Delta)^{\ell/2}u\right\|_{L^p([0,T];L^q(\R^d))}. 
\end{equation*}
We construct the solution to \eqref{eq:gen} by a fixed point
argument. Set
\begin{equation*}
  \Phi(u)(t) = S(t)u_0- i\l\int_0^t
  S(t-\tau)\(|u(\tau)|^{2\si}u(\tau)\)d\tau. 
\end{equation*}
We prove that for
$T\in]0,T_0]$ sufficiently small, $\Phi$ is a 
contraction on some ball of $X_T$ centered at the origin. In view of
\eqref{eq:stri} and Minkowski inequality,
\begin{align*}
  \left\|\Phi(u)\right\|_{X_T} &\lesssim \|u_0\|_{H^s(\R^d)} +
  \int_0^t\left\| |u(\tau)|^{2\si}u(\tau) \right\|_{H^s(\R^d)} d\tau\\
&\lesssim \|u_0\|_{H^s(\R^d)} +
  \int_0^t\left\|
    u(\tau)\right\|_{L^\infty(\R^d)}^{2\si}\left\|u(\tau)
  \right\|_{H^s(\R^d)} d\tau \\
&\lesssim \|u_0\|_{H^s(\R^d)} +
  T^{\gamma}\left\|
    u\right\|_{L^p([0,T];L^\infty(\R^d))}^{2\si}\left\|u
  \right\|_{L^\infty([0,T];H^s(\R^d))}, 
\end{align*}
with $\gamma = 1-2\si/p>0$. Therefore,
\begin{equation*}
\left\|\Phi(u)\right\|_{X_T}  \le C\||u_0\|_{H^s(\R^d)} +
  C T^{\gamma}\left\|
    u\right\|_{X_T}^{2\si+1}.
\end{equation*}
Similarly, 
\begin{equation*}
\left\|\Phi(u)-\Phi(v)\right\|_{X_T}  \le 
  C T^{\gamma}\(\left\|
    u\right\|_{X_T}^{2\si}+\left\|
    v\right\|_{X_T}^{2\si}\)\left\|
    u-v\right\|_{X_T}.
\end{equation*}
This yields the local well-posedness result stated in
Proposition~\ref{prop:LWP}. 
\end{proof}

\section{Ill-posedness}
\label{sec:ill}

The key remark is that all the cases of Proposition~\ref{prop:ill}
boil down to an ordinary differential equation mechanism. Denote by
$v$ the solution to  
\begin{equation*}
  i\d_t v = \l|v|^{2\si}v;\quad v_{\mid t=0}=u_0. 
\end{equation*}
When $P(\xi)=c\cdot \xi$, we have
\begin{equation*}
  u(t,x) = v\(t, x+ct\),
\end{equation*}
and when $P(\xi)=c$, we have
\begin{equation*}
  u(t,x) =v(t,x)e^{ict},
\end{equation*}
so it suffices to prove Proposition~\ref{prop:ill} in the case $P=0$. 
For fixed $x\in \R^d$, we have
\begin{equation}\label{eq:v}
  v(t,x) = u_0(x)e^{-i\l t|u_0(x)|^{2\si}}.
\end{equation}
The idea is then that $H^{d/2}(\R^d)$ is not an algebra. 
\smallbreak

Consider 
  \begin{equation}\label{eq:sing}
    u_0(x) = \delta\times\(\log \frac{1}{|x|}\)^{\alpha}\chi \(|x|^2\), \quad
    x\in \R^d,
  \end{equation}
with $\chi\in C_0^\infty(\R)$, $\chi=1$ near the origin, and ${\rm
  supp }\chi\subset ]-1,1[$. We compute
\begin{equation*}
  \nabla v(t,x) = e^{-i\l t|u_0(x)|^{2\si}}\nabla u_0(x) -2i\si\l t
  |u_0(x) |^{2\si}\nabla u_0(x). 
\end{equation*}
We split the proof into three cases: for $d=2$, the proof is
straightforward, for $d\ge 4$ even, the proof is similar but we omit
the details of computations, and for $d$ odd, we simply sketch the
argument. 

\subsection*{Case $d=2$}
First, $u_0\in H^1(\R^2)$ provided that $\alpha<1/2$. Now
Proposition~\ref{prop:ill} follows if we can choose $\alpha<1/2$ so
that $|u_0|^{2\si}\nabla u_0\not \in L^2(\R^2)$. Near the origin, we
have, leaving out the constants,
\begin{equation*}
  \left\lvert |u_0(x) |^{2\si}\nabla u_0(x)\right\rvert^2 \approx 
  \frac{1}{|x|^2}\(\log |x|\)^{4\alpha\si +  2\alpha-2}.
\end{equation*}
The right hand side fails to be in $L^1_{\rm loc}(\R^2)$ if we impose
$4\alpha\si +  2\alpha \ge 1$. So Proposition~\ref{prop:ill} follows, with
\begin{equation*}
  \frac{1}{4\si+2}\le \alpha<\frac{1}{2}.
\end{equation*}

\subsection*{Case $d\ge 4$ even}
The argument is the same as in the case $d=2$, with more computations
that we simply sketch. We check by induction that
for $k\ge 1$, there exist coefficients $(\beta_{jk})_{1\le j\le k}$
such that near the origin,
\begin{equation*}
  \d_r^ku_0(x) =\frac{1}{r^k}\sum_{j=1}^k \beta_{jk}\(\log
  \frac{1}{|x|}\)^{\alpha-j} ,\quad\text{with } \beta_{1k}=
  (-1)^{k-1}(k-1)!\alpha. 
\end{equation*}
Therefore, the asymptotic behavior of $\d_r^ku_0$ near the origin is given
by:
\begin{equation*}
  \d_r^ku_0(x)\Eq r 0 (-1)^{k-1}(k-1)!\frac{\alpha}{r^k}\(\log
  \frac{1}{|x|}\)^{\alpha-1}.
\end{equation*}
Like in the case $d=2$, $u_0\in H^{d/2}(\R^d)$ provided that
$\alpha<1/2$. We compute, for $t>0$, and as $x\to 0$:
\begin{equation}\label{eq:vasym}
  |\d_r^k v(t,x)| =\frac{1}{r^k}\(c_k(t)\(\log
  \frac{1}{|x|}\)^{\gamma_k}+
\O\(\(\log \frac{1}{|x|}\)^{\gamma_k-\omega_k}\)\), 
\end{equation}
with $c_k(t)> 0$, $\omega_k>0$, and
\begin{equation*}
  \gamma_k = \max\( (2\si k+1)\alpha -k,(2\si+1)\alpha-1\). 
\end{equation*}
For $t>0$, $v(t,\cdot)\not \in
H^{d/2}(\R^d)$ if, for $k=d/2$, the first term in \eqref{eq:vasym} is
almost in $L^2_{\rm loc}(\R^d)$, but not quite: we choose $\alpha$ so
that $2\gamma_{d/2}=-1$. We find (like for $d=2$)
\begin{equation*}
  \alpha = \frac{1}{4\si+2},
\end{equation*}
which is consistent with the requirement
$\alpha<1/2$. Thus, the first term
in \eqref{eq:vasym} is not in 
$L^2_{\rm loc}(\R^d)$ due to a logarithmic 
divergence, while the remainder term is in $L^2_{\rm loc}(\R^d)$,
since $\omega_k>0$. 

\subsection*{The case when $d$ is odd} We keep $u_0$ of the same form
as in even dimensions, since we have found a value for $\alpha$ which
does not depend on $d$ even:
\begin{equation*}
    u_0(x) = \delta\times\(\log \frac{1}{|x|}\)^{1/(4\si+2)}\chi \(|x|^2\), \quad
    x\in \R^d.
  \end{equation*}
Recall the characterization of $H^s(\R^d)$ when $s\in ]0,1[$: a
function $f\in L^2(\R^d)$ belongs to $H^s(\R^d)$, $s\in ]0,1[$, if and
only if
\begin{equation*}
  \iint_{\R^d\times\R^d}
  \frac{\left|f(x)-f(x+y)\right|^2}{|y|^{d+2s}}dxdy<\infty. 
\end{equation*}
When $d=1$, we check that $u_0\in H^{1/2}(\R)$. We can also
check that for $t>0$, $v(t,\cdot)\not \in H^{1/2}(\R)$. 
\smallbreak

When $d\ge 3$, we compute $\d_r^ku_0$ and $\d_r^k v$ in the same
fashion as above, and check that
\begin{equation*}
  \nabla^{[d/2]}u_0\in H^{1/2}(\R^d),\quad \text{and for }t>0,\quad
  \nabla^{[d/2]}v(t,\cdot)\not\in H^{1/2}(\R^d) .
\end{equation*}
We leave out the details, since the
technicalities are more involved than in the even dimensional case,
and we believe that proving 
Proposition~\ref{prop:ill} in details is not worth
such an effort.  

\subsection*{Acknowledgement}
The author is grateful to Valeria Banica for valuable comments
on this work. 

\bibliographystyle{siam}
\bibliography{disp}

\begin{thebibliography}{10}

\bibitem{Ant08}
{\sc R.~Anton}, {\em Strichartz inequalities for {L}ipschitz metrics on
  manifolds and nonlinear {S}chr\"odinger equation on domains}, Bull. Soc.
  Math. France, 136 (2008), pp.~27--65.

\bibitem{BAKS00}
{\sc M.~Ben-Artzi, H.~Koch, and J.-C. Saut}, {\em Dispersion estimates for
  fourth order {S}chr\"odinger equations}, C. R. Acad. Sci. Paris S\'er. I
  Math., 330 (2000), pp.~87--92.

\bibitem{BAKS03}
\leavevmode\vrule height 2pt depth -1.6pt width 23pt, {\em Dispersion estimates
  for third order equations in two dimensions}, Comm. Partial Differential
  Equations, 28 (2003), pp.~1943--1974.

\bibitem{BAS99}
{\sc M.~Ben-Artzi and J.-C. Saut}, {\em Uniform decay estimates for a class of
  oscillatory integrals and applications}, Differential Integral Equations, 12
  (1999), pp.~137--145.

\bibitem{BSS08}
{\sc M.~Blair, H.~Smith, and C.~Sogge}, {\em On {S}trichartz estimates for
  {S}chr\"odinger operators in compact manifolds with boundary}, Proc. Amer.
  Math. Soc., 136 (2008), pp.~247--256.

\bibitem{BoSa10}
{\sc J.~Bona and J.-C. Saut}, {\em Dispersive blow-up {II}.
  {S}chr\"odinger-type equations, optical and oceanic rogue waves}, Chin. Ann.
  Math. Ser. B, 31 (2010), pp.~793--818.

\bibitem{BGT}
{\sc N.~Burq, P.~G{\'e}rard, and N.~Tzvetkov}, {\em Strichartz inequalities and
  the nonlinear {S}chr\"odinger equation on compact manifolds}, Amer. J. Math.,
  126 (2004), pp.~569--605.

\bibitem{BGTENS}
\leavevmode\vrule height 2pt depth -1.6pt width 23pt, {\em Multilinear
  eigenfunction estimates and global existence for the three dimensional
  nonlinear {S}chr\"odinger equations}, Ann. Sci. \'Ecole Norm. Sup. (4), 38
  (2005), pp.~255--301.

\bibitem{CW90}
{\sc T.~Cazenave and F.~Weissler}, {\em The {C}auchy problem for the critical
  nonlinear {S}chr\"odinger equation in {$H^s$}}, Nonlinear Anal., 14 (1990),
  pp.~807--836.

\bibitem{CCT2}
{\sc M.~Christ, J.~Colliander, and T.~Tao}, {\em Ill-posedness for nonlinear
  {S}chr\"odinger and wave equations}.
\newblock Archived as \url{http://arxiv.org/abs/math/0311048}.

\bibitem{Cui05}
{\sc S.~Cui}, {\em Pointwise estimates for a class of oscillatory integrals and
  related {$L^p$}-{$L^q$} estimates}, J. Fourier Anal. Appl., 11 (2005),
  pp.~441--457.

\bibitem{Cui06}
\leavevmode\vrule height 2pt depth -1.6pt width 23pt, {\em Pointwise estimates
  for oscillatory integrals and related {$L^p$}-{$L^q$} estimates. {II}.
  {M}ultidimensional case}, J. Fourier Anal. Appl., 12 (2006), pp.~605--627.

\bibitem{CuGu07}
{\sc S.~Cui and C.~Guo}, {\em Well-posedness of higher-order nonlinear
  {S}chr\"odinger equations in {S}obolev spaces {$H^s(\mathbb R^n)$} and
  applications}, Nonlinear Anal., 67 (2007), pp.~687--707.

\bibitem{DeFa09}
{\sc A.~Debussche and E.~Faou}, {\em Modified energy for split-step methods
  applied to the linear {S}chr\"odinger equation}, SIAM J. Numer. Anal., 47
  (2009), pp.~3705--3719.

\bibitem{GV95}
{\sc J.~Ginibre and G.~Velo}, {\em Generalized {S}trichartz inequalities for
  the wave equation}, J. Funct. Anal., 133 (1995), pp.~50--68.

\bibitem{HaNa08}
{\sc N.~Hayashi and P.~I. Naumkin}, {\em Asymptotic properties of solutions to
  dispersive equation of {S}chr\"odinger type}, J. Math. Soc. Japan, 60 (2008),
  pp.~631--652.

\bibitem{KPV91}
{\sc C.~Kenig, G.~Ponce, and L.~Vega}, {\em Oscillatory integrals and
  regularity of dispersive equations}, Indiana Univ. Math. J., 40 (1991),
  pp.~33--69.

\bibitem{KPV94}
\leavevmode\vrule height 2pt depth -1.6pt width 23pt, {\em Higher-order
  nonlinear dispersive equations}, Proc. Amer. Math. Soc., 122 (1994),
  pp.~157--166.

\bibitem{KPV01}
\leavevmode\vrule height 2pt depth -1.6pt width 23pt, {\em On the ill-posedness
  of some canonical dispersive equations}, Duke Math. J., 106 (2001),
  pp.~617--633.

\bibitem{NaOz98}
{\sc M.~Nakamura and T.~Ozawa}, {\em Nonlinear {S}chr\"odinger equations in the
  {S}obolev space of critical order}, J. Funct. Anal., 155 (1998),
  pp.~364--380.

\bibitem{Pau07}
{\sc B.~Pausader}, {\em Global well-posedness for energy critical fourth-order
  {S}chr\"odinger equations in the radial case}, Dyn. Partial Differ. Equ., 4
  (2007), pp.~197--225.

\bibitem{Pau09}
\leavevmode\vrule height 2pt depth -1.6pt width 23pt, {\em The cubic
  fourth-order {S}chr\"odinger equation}, J. Funct. Anal., 256 (2009),
  pp.~2473--2517.

\bibitem{TaoDisp}
{\sc T.~Tao}, {\em Nonlinear dispersive equations}, vol.~106 of CBMS Regional
  Conference Series in Mathematics, Published for the Conference Board of the
  Mathematical Sciences, Washington, DC, 2006.
\newblock Local and global analysis.

\end{thebibliography}

 \end{document}